\documentclass[10pt,letter,reqno]{amsart}
\usepackage{amsfonts,amsthm,color,epsfig,latexsym,amsmath,amssymb,amscd,amsmath,epsf,graphicx,rotating}

\newtheorem{dfn}{Definition}
\newtheorem{thm}{Theorem}
\newtheorem{lemma}{Lemma}
\newtheorem{corollary}{Corollary}
\newtheorem{prop}{Proposition}
\newtheorem{remark}{Remark}

\renewcommand{\qed}{$\blacksquare$}
\newtheorem{thma}{Theorem A}

\newtheorem{thmb}{Theorem B}

\newcommand{\amoeba}{\mathrm{Amoeba}}
\newcommand{\anewt}{\mathrm{ArchNewt}}
\newcommand{\newt}{\mathrm{Newt}}
\newcommand{\Log}{\mathrm{Log}}
\newcommand{\N}{\mathbb{N}}
\newcommand{\C}{\mathbb{C}}
\newcommand{\R}{\mathbb{R}}
\newcommand{\Rn}{\R^n}
\newcommand{\Cs}{\C^*}
\newcommand{\ii}{I\!I} 
\newcommand{\iip}{I\!I^\geq}
\newcommand{\sip}{S\!I^\geq}
\newcommand{\si}{S\!I} 
\newcommand{\rr}{R\!R}
\newcommand{\rs}{S\!S}

\newcommand{\al}{\alpha}

\newcommand{\la}{\lambda}

\newcommand{\Z}{\mathbb{Z}}

\newcommand{\ga}{\gamma}
\newcommand{\eps}{\varepsilon}

\begin{document}
\numberwithin{equation}{section}

\title[New Multiplier Sequences via Discriminant Amoebae]
{\mbox{}\\
\vspace{-1.5in}
New Multiplier Sequences via Discriminant Amoebae}  

\author[M.\ Passare]{Mikael Passare}
\address{Department of Mathematics, Stockholm University, SE-106 91, Stockholm,
            Sweden}
\email{passare@math.su.se}

\author [J.\ M.\ Rojas]{J.\ Maurice Rojas}
\address{
Department of Mathematics, 
Texas A\&M University, 3368 TAMU,  College Station, Texas 77843,  USA}
\email{rojas@math.tamu.edu}

\author[B.\ Shapiro]{Boris Shapiro}
\address{Department of Mathematics, Stockholm University, SE-106 91, Stockholm,
            Sweden}
\email{shapiro@math.su.se}

\date{\today}

\begin{abstract}
In their classic 1914 paper,  
Pol\'ya and Schur introduced and characterized two types of  
linear operators acting diagonally on the monomial basis of 
$\R[x]$, sending real-rooted polynomials (resp.\ polynomials 
with all nonzero roots of the same sign) to real-rooted polynomials. 
Motivated by fundamental properties of amoebae and discriminants discovered 
by Gelfand, Kapranov, and Zelevinsky, we introduce two new natural classes of  
polynomials and describe diagonal operators preserving these new classes. A 
pleasant circumstance in our description is that these classes have a 
simple explicit description, one of them coinciding with 
the class of log-concave sequences. 
\end{abstract}

\maketitle

\vspace{-1cm} 
\section{Introduction}
\label{sec:int}  
The theory of {\em linear preservers} (linear operators preserving certain 
families of matrices or polynomials) is a widely developed and active area of 
mathematics (see, e.g., \cite{SU} and the references therein). 
Linear preservers have found applications in many 
areas such as approximation theory, probability theory, and 
statistics (see, e.g., \cite{karlin}), and have even been 
used to give interesting reformulations of the Riemann Hypothesis 
\cite{csordas}. One of the most classical instances of the theory of linear 
preservers occurs in the setting of real-rooted polynomials, initiated in the 
late 19th century by Laguerre and Hermite.  

Given a sequence of real numbers  $\ga=\{\ga_{j}\}^\infty_{j=0}$  
consider the linear operator $T_{\ga}: \R[x]\to \R[x]$ acting on  
each $x^j$ by multiplication by $\ga_j$. We refer to such a $T_\ga$ as 
the {\em diagonal operator} corresponding to $\ga$. Let 
$\rr\subset\!\R[x]$ denote the collection of polynomials all of whose 
complex roots are real, i.e., {\em \underline{r}eal-\underline{r}ooted} 
polynomials. Following  
\cite{PS} we 
call $\gamma$ a {\em multiplier sequence} ({\em``Faktorenfolge"}) 
{\em of the first kind} if $T_\ga(\rr)\!\subseteq\!\rr$. Similarly, 
let $\rs$ denote the subset of $\rr$ consisting of polynomials $p$ whose 
nonzero roots (all real, by assumption) are all of the \underline{s}ame 
\underline{s}ign. 
A {\em multiplier sequence of the 2nd kind} is then a $\gamma$ with 
$T_\ga(\rs)\!\subseteq\!\rr$.  

The following result of Pol\'ya and Schur is fundamental. 
\begin{thma}\label{th:PS} \cite{PS} 
Let $\ga=\{\ga_j\}^\infty_{j=0}$ be a sequence of real numbers
and $T_\ga : \R[x] \to  \R[x]$ the corresponding diagonal  operator. 
Then: 
\begin{enumerate} 
\item[(i)]{$\ga$ is a multiplier sequence of the $1$st kind
(i.e., $T_\ga(\rr)\!\subseteq\!\rr$)
iff for all $n\!\in\!\N$ we have $T_\ga((1+x)^n)\!\in\!\rs$.} 
\item[(ii)]{$\ga$ is a multiplier sequence of the $2$nd kind
(i.e., $T_\ga(\rs)\!\subseteq\!\rr$)
iff for all $n\!\in\!\N$ we have $T_\ga((1+x)^n)\!\in\!\rr$. \qed} 
\end{enumerate}
\end{thma} 
\begin{remark} 
Pol\'ya and Schur also obtained a transcendental 
characterization in terms of the generating function 
$\Phi_\ga(x)=\sum\limits^\infty_{k=0} \frac{\ga_k}{k!} x^k$. 
\end{remark} 

There exist obvious versions of these notions for polynomials of bounded
degree. In particular, a sequence $\ga=(\ga_0,\ga_1,...,\ga_k)$  will be 
referred to as a {\em multiplier sequence of length $k+1$} or simply a 
{\em finite multiplier sequence} if it has the above mentioned properties when 
acting on the linear space $\R_k[x]$ of real polynomials of degree at most 
$k$. In particular, we define $\rr_k:=\rr\cap\R_k[x]$ and 
$\rs_k:=\rs\cap \rr_k$. 

Craven and Csordas proved 60 years later that for a finite length 
multiplier\linebreak 
\scalebox{.93}[1]{sequence $\ga$, checking whether $\ga$ is of first or second 
kind can be reduced to checking the}\linebreak 
\scalebox{.95}[1]{image of just one polynomial under $T_\ga$ 
(see \cite[Thm.\ 3.7]{CC2} and \cite[Thm.\ 3.1]{CC1}).}   
\begin{thmb}\label{th:MS}
Let $\ga=(\ga_0,\ldots,\ga_k)$ and $T_\ga$ the corresponding diagonal 
operator. Then for all $k\!\in\!\N$, we have: 
\begin{enumerate}
\item[(i)]{$T_\ga(\rr_k)\!\subseteq\!\rr_k$ iff $T_\ga\!\left((1+x)^k\right)\in 
\rs$.}  
\item[(ii)]{$T_\ga\!\left(\rs_k\right)\!\subseteq\!\rr_k$ iff 
$T_\ga\!\left((1+x)^k\right)\in \rr$.}  
\end{enumerate}
\end{thmb}
\begin{remark} 
While Assertion (i) is merely a rewording of \cite[Thm.\ 3.7]{CC2},\linebreak 
Assertion (ii) appears to be new and follows upon a closer examination 
of Section 3 of \cite{CC2}. 
\end{remark} 

Letting $q(x)\!:=\!x^m(1+x)^2$, note that $q\in\rs\subsetneqq\rr$ and $q$ 
has $-1$ as a root of multiplicity $2$. 
It then follows that if one decreases the coefficient of $x^{m+1}$ 
in $q$ (and leaves the coefficients of $x^m$ and $x^{m+2}$ fixed) 
then the resulting polynomial has non-real roots. With a little more work one 
then easily concludes that any multiplier sequence $\ga=(\gamma_0,\gamma_1,
\ldots)$ of first or second kind must satisfy {\em Tur\'an's Inequalities} 
(see, e.g., \cite{CVV} and  \cite[Problem 4.8]{CC3}):  
$\ga^2_j\geq\ga_{j-1}\ga_{j+1}$ for all 
$j\geq 2$. Since we can naturally identify any finite 
multiplier sequence $(\ga_0,\ldots,\ga_k)$ with the 
infinite sequence $(\ga_0,\ldots,\ga_k,0,0,\ldots)$ the 
Tur\'an Inequalities clearly hold for finite length multiplier sequences 
(of first or second kind) as well. The converse fails, however, as 
can be easily seen by perturbing the nonzero coefficients of $x^m(1+x)$ 
instead. The occurence of roots of multiplicity $>\!1$ here is one reason 
it is natural to start thinking of discriminants (see also Figures 1 and 2 
below). 

\begin{remark} 
Positive sequences satisfying Tur\'an's inequalities are called 
{\em log-concave} and find frequent applications in combinatorics. 
An analoguous notion with the coefficients weighted by binomial 
coefficients is known as {\em ultra log-concavity} \cite{Li,KSh}. 
\end{remark} 

We will return to $x^m(1+x)$ momentarily but observe now that the polynomial 
$x^m(1+x)^2$ has the following special property: all polynomials 
obtained by arbitrary sign flips of its coefficients also belong to $\rr$. 
\begin{dfn}
A real polynomial $p$ is called {\em sign-independently 
real-rooted} if $p$ is real-rooted and all polynomials obtained by  
arbitrary sign flips of the coefficients of $p$ are real-rooted as well. 
We let $\si$ denote the set of all \underline{s}ign-\underline{i}ndependently 
real-rooted polynomials and $\sip$ denote the subset of $\si$ 
consisting of polynomials with all coefficients nonnegative. 
Finally, we call $\ga$ a {\em multiplier sequence of the 
$3$rd kind} if $T_\ga\!\left(\sip\right)\!\subseteq\!\rr$.  
\end{dfn}

Clearly, $\sip\!\subsetneqq\!\si\!\subsetneqq\!\rs\!\subsetneqq\!\rr$. 
Another simple example of a sign-independently real-rooted polynomial is 
$x^m(1+x)$ and less trivial examples can be found in Section \ref{sub:real}. 
Similar to our earlier development we define $\si_k\!:=\!\si\cap \R_k[x]$ 
and $\sip_k\!:=\!\sip\cap\R_k[x]$. The sets $\sip_3$, $\rs_3$, and $\rr_3$ are 
illustrated in Figure 2 below. 

\medskip
Our main results are summarized by the following $2$ theorems and 
a corollary. 
\begin{thm}\label{th:main} $\gamma$ is a 
multiplier sequence of the third kind (finite or infinite) 
iff it is log-concave, i.e., $T_\ga\!\left(x^n(1+x)^2\right)\in\rr$ for all 
$n\!\in\!\N$. Moreover, any such $\gamma$ satisfies 
$T_\ga\!\left(\sip\right)\!\subseteq\!\sip$.  
\end{thm} 

\begin{corollary} 
If $p(x)=a_0+a_1x+\cdots+a_kx^k\in\sip_k$ then 
$a^2_\nu\geq 4a_{\nu-1}a_{\nu+1}$ for all $\nu\in\{1,\ldots,k-1\}$, 
and any truncated polynomial $a_mx^m+a_{m+1}x^{m+1}+\cdots+a_nx^n$ 
(for $0\leq m<n\leq k$) has all its nonzero roots negative. 
\end{corollary} 

\noindent 
\begin{picture}(200,220)(0,-37)
\put(150,0){\epsfig{file=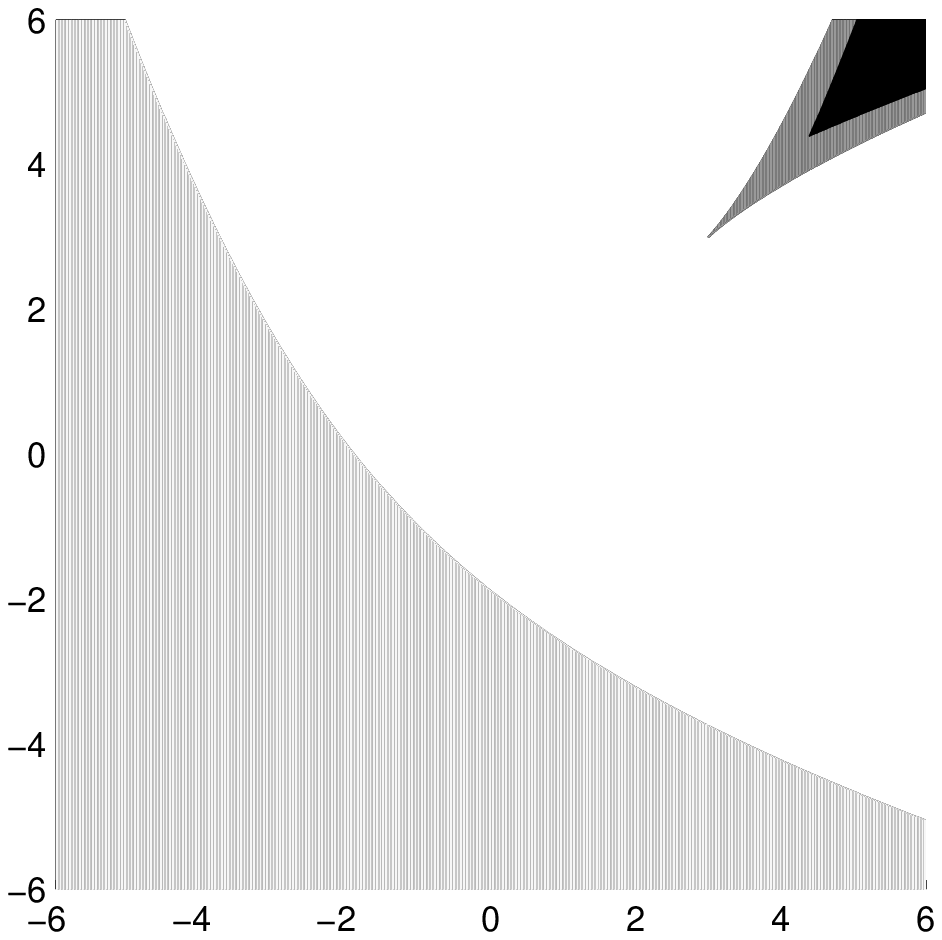,height=2.6in}} 
\put(-40,0){\epsfig{file=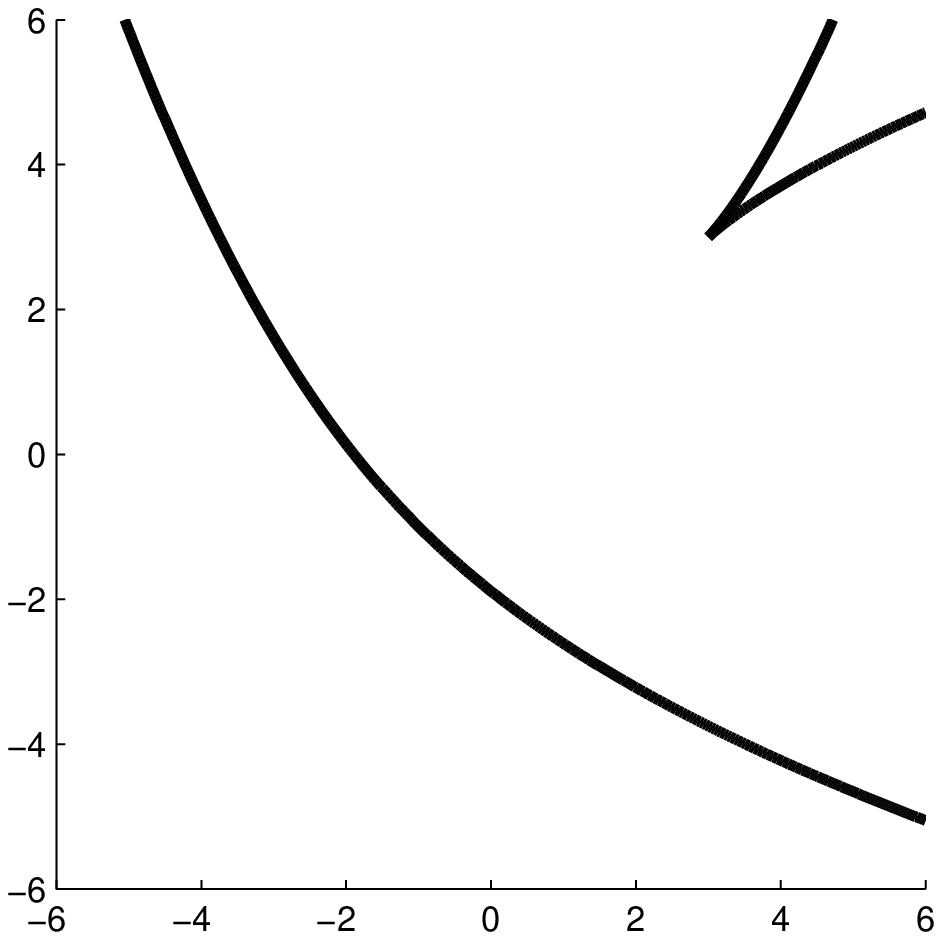,height=2.6in}}
\put(13,176){{\small $b$}} 
\put(30,60){{\small {\bf 3 real, with}}} 
\put(30,50){{\small {\bf different signs}}} 
\put(80,110){{\small {\bf only 1 real}}} 
\put(144,150){\begin{rotate}{45}
{\tiny {\bf 3 negative}} \end{rotate}}  
\put(167,20){$a$}  
\end{picture}

\vspace{-1.8cm} 
\noindent 
\begin{minipage}[t]{2.4in} 
{\scriptsize {\sc Figure 1}: The discriminant variety of the 
family\linebreak  

\vspace{-.3cm} 
$1+ax+bx^2+x^3$ separates the coefficient space\linebreak  

\vspace{-.45cm} 
into regions according to the number of real roots.} 
\end{minipage} 
\hspace{.3in}
\begin{minipage}[t]{2.2in} 
{\scriptsize {\sc Figure 2}: Corresponding slices of $\sip_3$, 
$\rs_3$,\linebreak

\vspace{-.3cm} 
and $\rr_3$: $\sip_3$ is in black, $\sip_3\!\subsetneqq\!
\rs_3\!\subsetneqq\!\rr_3$,\linebreak 

\vspace{-.45cm} 
and the complement of $\rr_3$ is white. } 
\end{minipage} 

\medskip 
\noindent
Davenport and Pol\'ya observed earlier \cite{DaPo} that 
log-concave positive sequences 
form a semigroup with respect to the {\em Hadamard product} 
$(\ga_0,\ga_1,\ldots)\cdot (\ga'_0,\ga'_1,\ldots)\!:=\!
(\ga_0\ga'_0,\ga_1\ga'_1,\ldots)$. In particular, it will be fruitful 
to observe later that the image of such sequences under coordinate-wise 
logarithm forms a cone.  

More to the point, via {\em $A$-discriminant theory} 
\cite{GKZ}, we can reinterpret the sets $\sip_k$, $\rs_k$, and $\rr_k$ 
in terms of the complement of an important hypersurface associated 
to $k$. This point of view yields yet another new family of multiplier 
sequences, in some sense dual to $\sip$. \\ 
{\bf Definition 2.} {\em 
We define $\iip_k$ to be the set of those polynomials 
$p(x)=a_0+a_1x+\cdots+a_kx^k$ such that (i) $a_j\!\geq\!0$ for all 
$j$, (ii) $a_0,a_k\!>\!0$, (iii) $p$ has exactly $1$ or $0$ 
real roots according as $k$ is odd or even, (iv) for any polynomial $p^*$  
obtained from $p$ by multiplying any subset of the $a_i$ with $i\!\in\!\{1,
\ldots,k-1\}$ by $-1$, $p^*$ also has maximally many 
\underline{i}maginary roots in the sense of Condition (iii).}   

\noindent 
Note that for $k$ even, any polynomial $p\!\in\!\iip_k$ is positive 
on all of $\R$, and any $p^*$ obtained from $p$ (as in Condition (iv) above) 
is also positive on all of $\R$. \\
{\bf Theorem 2.}  
{\em A positive sequence $\ga\!:=\!(\ga_0,\ldots,\ga_k)$ satisfies 
$T_\ga\!\left(\iip_k\right)\!\subseteq\!\iip_k$ iff

\vspace{-.25cm}
\noindent 
$\ga^k_j\!\leq\!
\left(\frac{\ga_k}{\ga_0}\right)^j$ for all $j\!\in\!\{1,\ldots,k-1\}$.}  

\noindent 
Within the next section, 
we will see how $\si^\geq_k$ and $\ii^\geq_k$ correspond naturally 
to opposite connected components of a particular {\em amoeba} complement. 

\section{Background on Discriminants and Amoebae} 
The first ingredient to proving our main results is the 
following construction:\linebreak 
\scalebox{.94}[1]{Consider the map  
$\Log|\cdot|: (\Cs)^{k+1}\to \R^{k+1}$ sending 
${\bf a}\mapsto\bigl(\log|a_0|,\log|a_1|,\ldots,
\log|a_k|\bigr)$,}\linebreak   
where ${\bf a}=(a_0,a_1,...,a_k)\in 
(\Cs)^{k+1}$. Notice that $\Log|\cdot|$ maps $\R_+^{k+1}$ 
diffeomorphically onto $\R^{k+1}$ where $\R_+$ is the set of all positive 
real numbers. 

\addtocounter{footnote}{1} 
For any polynomial $q\in \C[a_0,...,a_k]$ one defines its 
{\em amoeba} $\amoeba(q)$ as the image of 
the complex algebraic hypersurface\\ 
\mbox{}\hfill 
$H_q:= \left\{{\bf a}=(\left. a_0,\ldots,a_k)\in(\Cs)^{k+1}\; \right| \; 
q({\bf a})=0\right\}$ \hfill\mbox{}\\ 
under $\Log|\cdot|$.
Recall also that the {\em Newton polytope} of $q(x):=\sum_{{\bf \al}\in A} 
c_{\bf \al} x^{\al}$, written $\newt(q)$, is the convex hull of$^1$ of 
$\{\al\in\Z^{k+1} \; | \; c_\al\!\neq\!0\}$, where the notation 
$x^{\al}:=x^{\al_0}_1 \cdots x^{\al_k}_k$
is understood.\footnotetext{i.e., smallest convex set containing...} 
There is a natural $1$-$1$ correspondence between unbounded 
connected components of the complement $\R^{k+1}\setminus \amoeba(q)$ and the 
vertices of $\newt(q)$.  
\begin{lemma} 
\label{lem:gkz} 
\cite[Prop.\ 1.7 \& Cor.\ 1.8, pp.\ 195--196]{GKZ} 
Suppose a polynomial $f \in \C[x_1,\ldots,x_n]$ has 
Newton polytope $P$ and $v$ is a vertex of $P$. Also let 
$C$ denote the closure of the cone of inner normals to $v$. Then there 
is a unique unbounded 
connected component $\Gamma$ of the complement to $\amoeba(f)$ 
containing a translate of the cone $C$. \qed 
\end{lemma} 

\noindent 
The cone $C$ above is also called the {\em recession cone} of $\Gamma$, since 
it consists of all translations $y\!\in\!\Rn$ with $y+\Gamma\subseteq\Gamma$. 

Let $\Delta_k$ denote the discriminant of the family of polynomials 
$a_0+\cdots+a_kx^k$, i.e., $\Delta_k\!\in\!\Z[a_0,\ldots,a_k]$ is the 
unique (up to sign) irreducible polynomial such that 
$a_0+\cdots+a_kx^k$ has a root of multiplicity $>\!1$ implies 
that $\Delta_k(a_0,\ldots,a_k)\!=\!0$. For instance, 
$\Delta_3:=-27a^2_0a^2_3+18a_0a_1a_2a_3+a^2_1a^2_2-4a_0a^3_2-4a^3_1a_3$.  
More generally, $\Delta_k$ can be computed using a number of arithmetic 
operations polynomial in $k$ (via a\linebreak 
\scalebox{.95}[1]{standard formula involving a $(2k-1)\times (2k-1)$ 
determinant), and is the special case}\linebreak 
\scalebox{.94}[1]{$A\!=\!\{0,\ldots,k\}$ of an {\em $A$-discriminant} 
(see \cite[Ch.\ 9 \& 12]{GKZ} for further background).}  

\vbox{
Amoebae of $A$-discriminants have a more refined structure.  
For example, the boundary of $\amoeba(\Delta_k)$ is contained in the image of 
the real part $H^\R_{\Delta_k}$ of the complex algebraic hypersurface 
$H_{\Delta_k}$ under $\Log|\cdot|$ (see Figure 3 below).  
The latter fact  motivates the following 
definition. 
\addtocounter{dfn}{1} 
\begin{dfn} 
For a complex algebraic hypersurface 
$H_q\subset \C^{k+1}$ given by\linebreak 
$q(a_0,a_1,...,a_k)=0$ we define its {\em 
complete reflection} 
$H^\dagger_q$ as the union of the $2^{k+1}$ hypersurfaces given by $q(\pm a_o, 
\pm a_1,..., \pm a_k)=0$ for all $2^{k+1}$ possible choices of signs of 
coordinates (see, e.g., Figure 4 below). 
\end{dfn} 
\begin{minipage}[t]{2.4in} 
\mbox{}\hspace{-.5cm}\epsfig{file=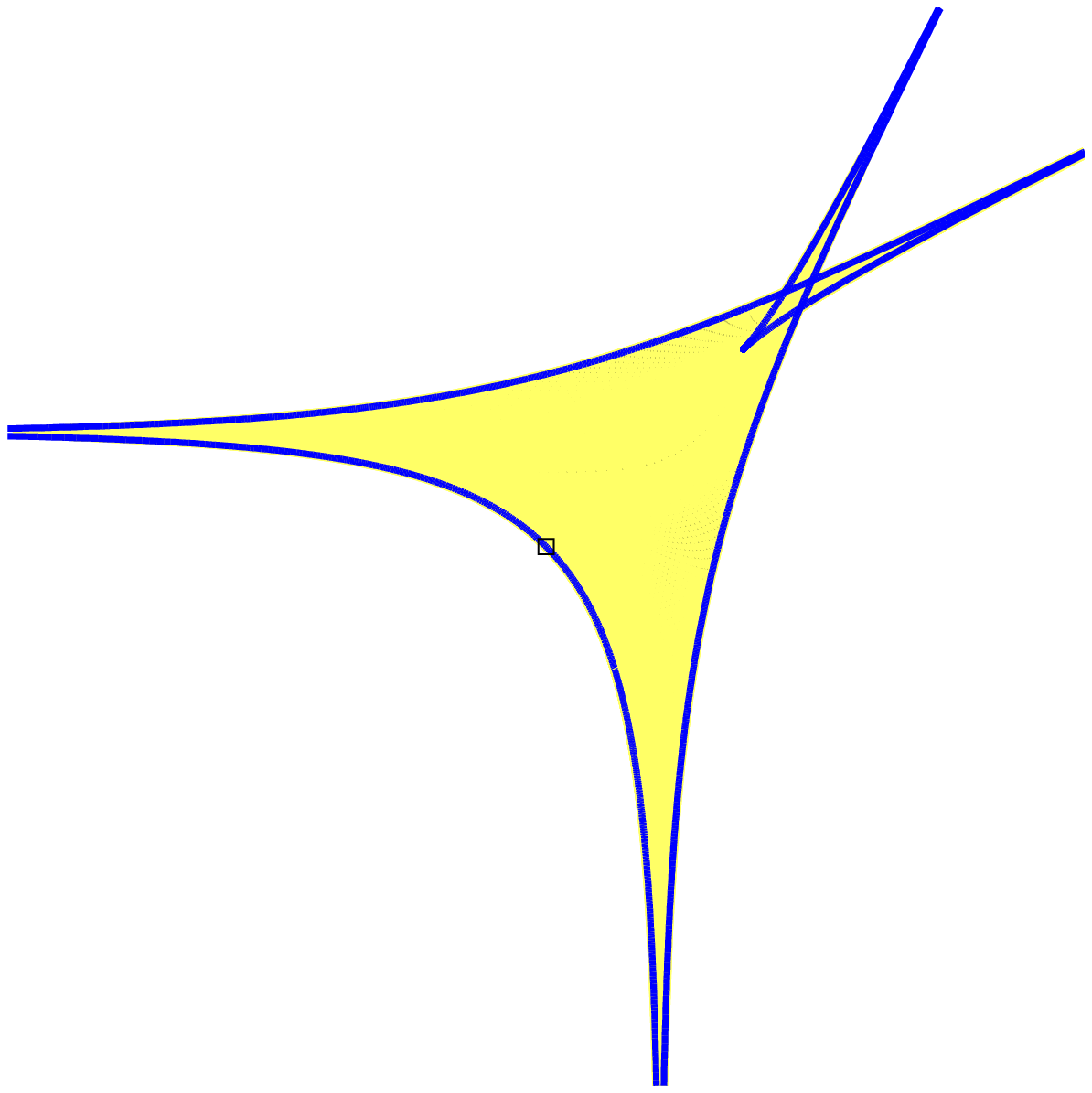,height=2in}

\vspace{-.4cm} 
{\scriptsize {\sc Figure 3:} The amoeba of the specialized cubic\linebreak 

\vspace{-.3cm}
discriminant $\Delta_3(1,a,b,1)$ (in yellow), and the\linebreak 

\vspace{-.45cm}
image of $H^\R_{\Delta_3(1,a,b,1)}$ under $\Log|\cdot|$ (in blue).}
\end{minipage} \hspace{.5cm} 
\begin{minipage}[t]{2.2in} 
\epsfig{file=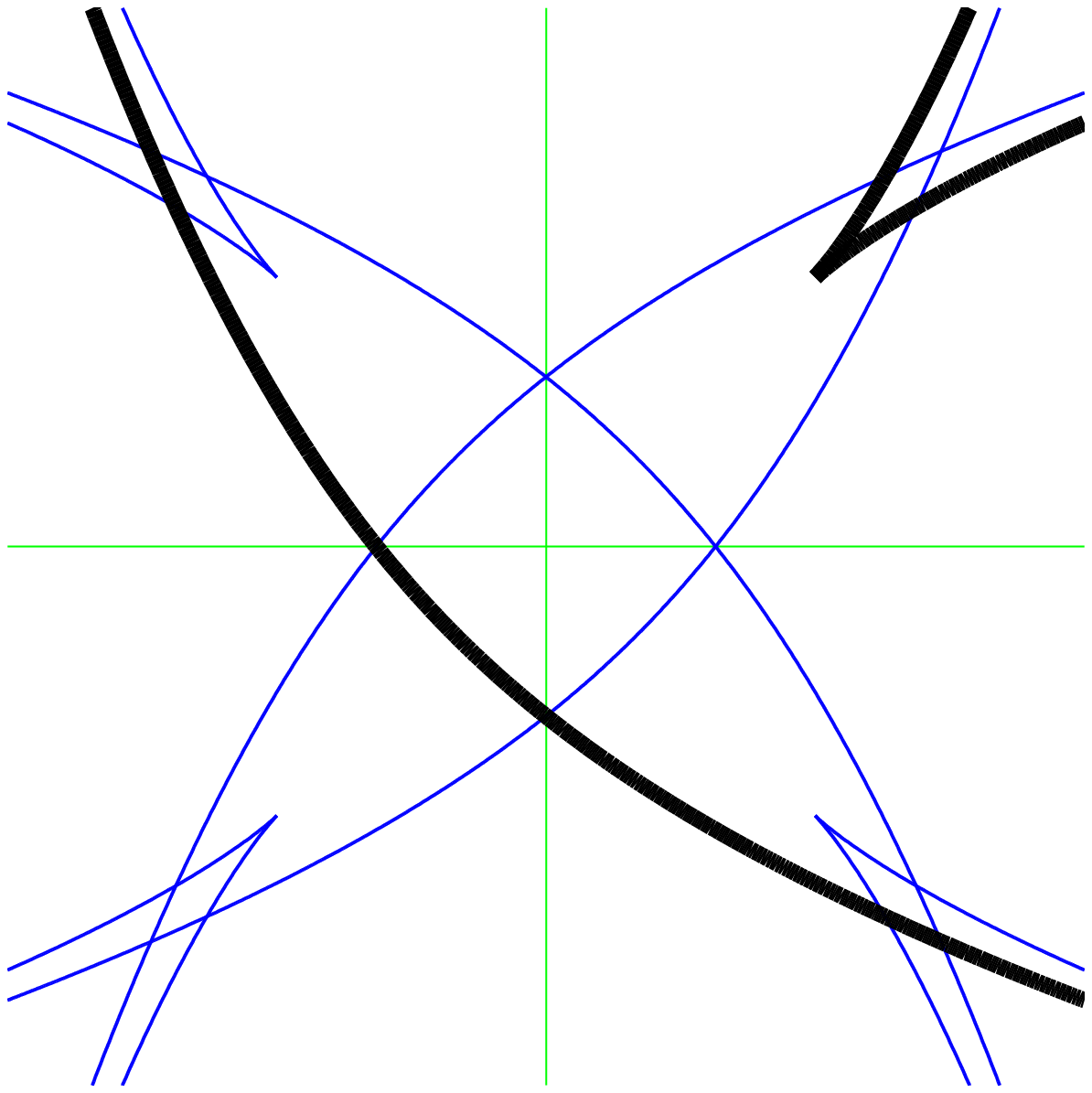,height=2in} 

\vspace{-.4cm} 
{\scriptsize {\sc Figure 4:} The real part of the discriminant\linebreak

\vspace{-.3cm}
variety of the family $1+ax+bx^2+x^3$ (bold)\linebreak

\vspace{-.45cm}
and its sign flips, i.e., $H^\dagger_{\Delta_3(1,a,b,1)}$.}
\end{minipage}} 

Consider the restriction of the real part 
$(H^\dagger_q)^\R$ of $H^\dagger_q$ to $\R_+^{k+1}$. Notice that 
by the above remark each connected component of $\R_+^{k+1}\setminus 
(H^\dagger_q)^\R $ is mapped  by $\Log|\cdot|$ diffeomorphically either 
onto a connected component of the complement $\R^{k+1}\setminus \amoeba(q)$ 
or onto $\amoeba(q)$ itself. One thus sees that $\amoeba(q)$ is the union of 
the images of some number of the latter connected components. 

Returning to $\Delta_k$, it is well known (see, e.g., \cite[pg.\ 271]{GKZ}) 
that $\Delta_k$ has the two homogeneities:\\ 
\mbox{}\hfill $\Delta_k(\lambda a_0,\lambda a_1,\lambda a_2, \ldots,\lambda
a_k)=\lambda^{2(k-1)}\Delta_k({\bf a})$\hfill\mbox{}\\
and\\
\mbox{}\hfill $\Delta_k(a_0,\lambda a_1,\lambda^2 a_2, \ldots, \lambda^k
a_k)=\lambda^{k(k-1)}\Delta_k({\bf a})$.\hfill \mbox{}\\
This immediately implies that $\newt(\Delta_k)$ 
has codimension at least $2$. In fact, the codimension is exactly $2$, 
and it is then easy to see that $\amoeba(\Delta_k)$ is an $\R^2$-bundle 
over a base that is an amoeba of smaller dimension. In particular, one 
can take the base to be the amoeba of $\Delta_k(1,a_1,\ldots,a_{k-1},1)$, 
thus explaining why our illustrations for $k\!=\!3$ are in the plane, as 
opposed to $\R^4$. 

There is also a combinatorial formula for the monomials in $\Delta_k$  
with exponents corresponding to vertices of $\newt(\Delta_k)$ (see 
\cite[pgs.\ 300 \& 302]{GKZ}). Namely, each such
vertex monomial corresponds to
a unique subdivision of the line segment $[0,k]$ into a collection of
segments $\{[0,k_1],[k_1,k_2],\ldots,[k_m,k]\}$, with 
integers $0<k_1<k_2<\ldots< k_m<k$. In particular, the finest subdivision
$\{[0,1],\ldots,[k-1,k]\}$ of $[0,k]$ into unit intervals is
associated with the monomial\\
\mbox{}\hfill 
$\pm a_1^2a_2^2\cdots a_{k-1}^2=\pm(a_1a_2\cdots a_{k-1})^2$,\hfill (1)
\hfill\mbox{}\\
whereas the second finest subdivisions, having one segment $[l-1,l+1]$ of
length two and all other
segments of unit length, correspond to the monomials\\
\mbox{}\hfill$\pm 4\,a_{l-1}a_l^{-2}a_{l+1}(a_1a_2\cdots a_{k-1})^2\,,\quad
l\in\{1,\ldots,k-1\}$.\hfill (2)\hfill\mbox{}\\ 
Moreover, thanks to Lemma \ref{lem:gkz}, we obtain a trinity of associations 
(see also Figure 5 below): \\
\begin{picture}(200,160)(0,-50)
\put(30,40){vertex monomials of $\Delta_k$} 
\put(75,55){\begin{rotate}{45}\scalebox{2}[1]{$\pmb{\longleftrightarrow}$}
\end{rotate}}  
\put(100,100){certain unbounded connected}
\put(230,80){\begin{rotate}{-45}\scalebox{2}[1]
{$\pmb{\longleftrightarrow}$}\end{rotate}}  
\put(90,90){components of the complement of $\amoeba(\Delta_k)$} 
\put(200,40){triangulations of $\{0,\ldots,k\}$} 
\put(140,40){\scalebox{2}[1]{$\pmb{\longleftrightarrow}$}} 
\end{picture} 

\vspace{-1.1in}
Combinatorially $\newt(\Delta_k)$ is a cube of
dimension $k-1$ and the monomials (2) represent the vertices
$v_0+e_{l-1}-2e_l+e_{l+1}$ neighboring the vertex $v_0=(0,2,2,\ldots,2,0)$
corresponding to the monomial (1).

\subsection{Archimedean Newton Polygons} 
An arguably more direct association between polynomials of degree $k$ 
and subdivisions of the point set $\{0,\ldots,k\}$ can be 
obtained via the {\em Archimedean Newton polygon}, which dates 
back to work of Ostrowski in the 1940s \cite[pp.\ 106 \& 132]{ostrowski}. 
This particular kind of Newton polygon further  
elucidates the connection between Theorems 1 and 2, and 
we use the appelation ``Archimedean'' 
to complement the {\em non-Archimedean} Newton
polygons coming from number theory and tropical geometry. 
\begin{dfn}
Given any polynomial $f(x)\!=\!a_0+a_1x+\cdots+a_kx^k$, its
{\em Archimedean Newton polygon}, written $\anewt(f)$, is the
convex hull of the finite point set $\{(i,-\log|a_i|)\; | \;
i\in\{0,\ldots,k\}\}$. We also call any edge of $\anewt(f)$ a
{\em lower} edge if it has an inner normal with positive last coordinate.
\end{dfn}

One can observe experimentally that there is a deep correlation between the
slopes of the lower edges of $\anewt(f)$ and the absolute values of the
roots of $f$.\linebreak
\scalebox{.98}[1]{In particular, paraphrasing in more modern language, 
Ostrowski proved remarkable}\linebreak 
\scalebox{.93}[1]{explicit bounds revealing how the slopes of the lower edges 
of $\anewt(f)$ approximate}\linebreak 
\scalebox{.95}[1]{the negatives of the logs of the norms of the 
roots of $f$ \cite[pp.\ 106 \& 132]{ostrowski}.}  

Even more directly, one notes that the lower hull of $\anewt(f)$ 
naturally associates, via orthogonal projection onto the first coordinate,  
a triangulation of $\{0,\ldots,k\}$ to $f$. 
In particular, it is easy to derive that the strict 
log-concavity\footnote{Strict
log-concavity for $(\ga_0,\ldots,\ga_k)$ simply means that
$\ga^2_j>\ga_{j-1}\ga_{j+1}$ for all $j\in\{1,\ldots,k-1\}$.}
of the sequence of coefficients of $f$ is nothing more than the condition that
$\anewt(f)$ have exactly $k-1$ lower edges. However, unless $\anewt(f)$ 
is sufficiently ``bowed'', a degree $k$ polynomial $f$ having 
$\anewt(f)$ with $k-1$ edges need {\em not} correspond to a point in 
the corresponding component $\Gamma$ of the complement of $\amoeba(\Delta_k)$. 

\vspace{.15cm}
\noindent
\begin{picture}(200,220)(0,-37)
\put(-40,0){\epsfig{file=cubicdiscamoeba.eps,height=2.6in}}
\put(140,160){\epsfig{file=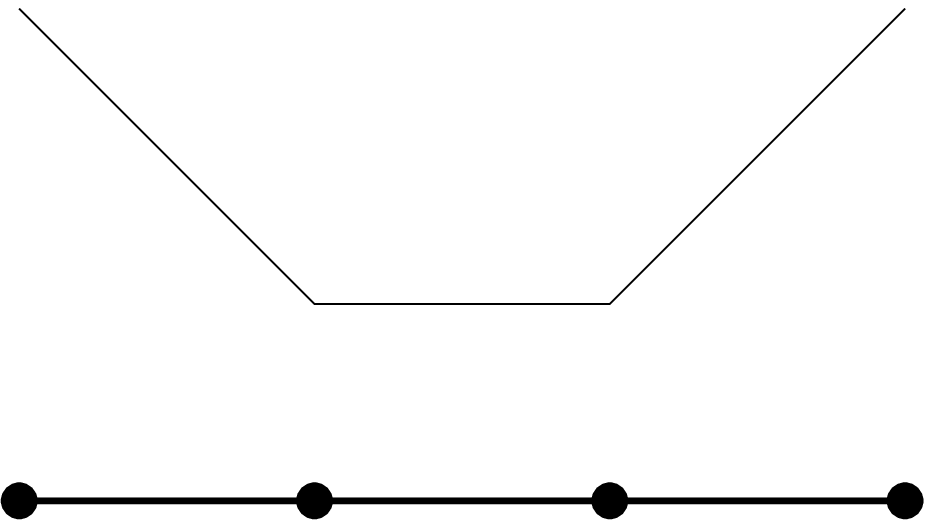,height=.25in}}  
\put(50,140){\epsfig{file=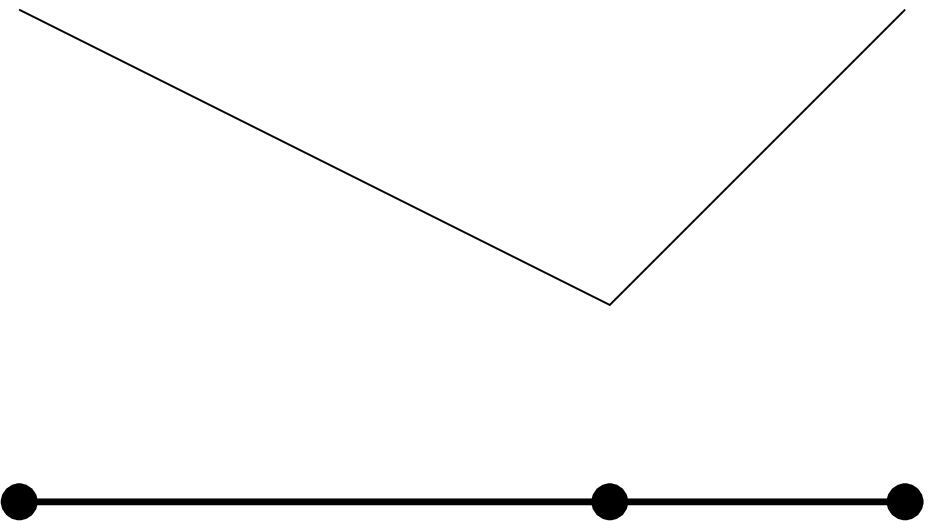,height=.3in}}  
\put(120,70){\epsfig{file=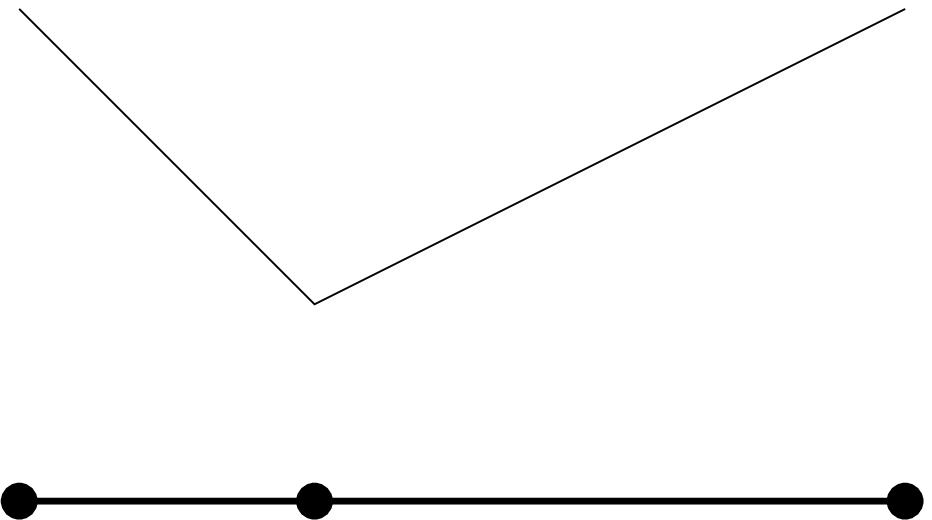,height=.3in}}  
\put(30,60){\epsfig{file=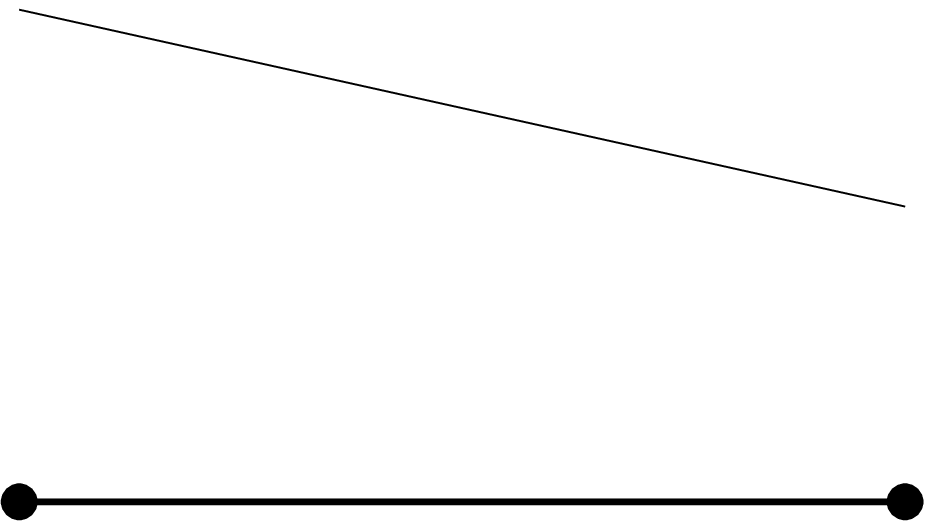,height=.3in}}  
\put(200,70){\epsfig{file=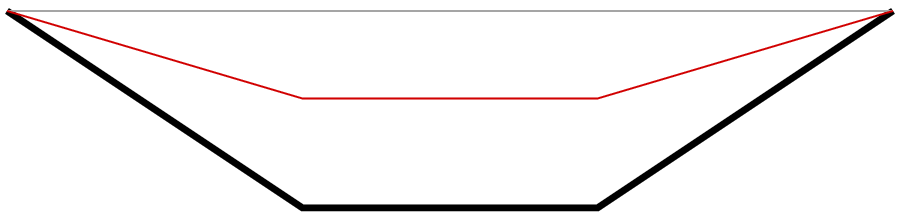,height=.5in}}
\end{picture}

\vspace{-1.8cm}
\noindent
\begin{minipage}[t]{2.4in}
{\scriptsize {\sc Figure 5}: Lower hulls of $\anewt(f)$,\linebreak 

\vspace{-.3cm} 
and associated subdivisions of $\{0,1,2,3\}$,\linebreak

\vspace{-.3cm} 
corresponding to the unbounded components\linebreak

\vspace{-.45cm} 
of the complement of $\amoeba(\Delta_3(1,a,b,1))$.} 
\end{minipage}
\hspace{.3in}
\begin{minipage}[t]{2.2in}
{\scriptsize {\sc Figure 6}: \textcolor{red}{$1+2.9x+2.9x^2+x^3$} does not 
lie\linebreak 

\vspace{-.3cm}
in the upper right component of Figure 5,\linebreak

\vspace{-.3cm}
but $\pmb{1+9x+9x^2+x^3}$ (with a more ``bowed''\linebreak 
  
\vspace{-.45cm}
lower hull for its $\anewt$) does... } 
\end{minipage} 

\noindent 
For instance, $1+2.9x+2.9x^2+x^3$ has only $1$ real root, but 
$1+9x+9x^2+x^3$ has $3$ real roots. 

As we will see in Lemma \ref{lem:push} of the next section, multiplier 
sequences can be used to make the lower hull of an $\anewt(f)$ more 
bowed. Similarly, the sequences
highlighted in Theorem 2 can clearly be identified with those $f$
having $\anewt(f)$ with exactly $1$ lower edge.
Thus, Theorem 1 (resp.\ Theorem 2) appears to relate 
maximal (resp.\ minimal) triangulations with polynomials having maximally
(resp.\ minimally) many real roots. 

\subsection{Supporting Results on Real-Rooted Polynomials} 
\label{sub:real} 
Using the notation $x_l=\log|a_l|$ we see that $\amoeba(\Delta_k)$ is the set 
of vectors $(x_0,...,x_k)\in\R^{k+1}$ such that
the torus $|a_0|=e^{x_0},\ldots, |a_k|=e^{x_k}$ intersects the
discriminant locus $H_{\Delta_k}$.  

\begin{prop}\label{pr:monom}The map $\Log|\cdot|$ is a 
diffeomorphism from $\sip_k$ to the connected component of the complement of 
$\amoeba(\Delta_k)$ corresponding to the monomial (1).
\end{prop}

\noindent 
The proof of this proposition is based on several additional statements. 
Along the way, we will also see some more examples of sign-independently 
real-rooted polynomials. 

First consider the vector $s\in\mathbb{N}^{k-1}$ given by\\
\mbox{}\hfill $s_j=\left(\left|\frac{k}{2}-j\right|+1\right) + 
\left(\left|\frac{k}{2}-j\right|+2\right) + \cdots + \frac{k}{2},
\quad  j\in\{1,\ldots,k-1\}$,\hfill\mbox{}\\
for $k$ even, and by\\
\mbox{}\hfill 
$s_j=\left(j-\frac{k-1}{2}\right) + \left(j+1-\frac{k-1}{2}\right) + \cdots 
+ \frac{k-1}{2}, \quad j\in\{1,\ldots, k-1\}$,\hfill\mbox{}\\
for $k$ odd.

The first few instances of $s$ are $(1)$ for $k=2$; $(1,1)$ for $k=3$;
$(2,3,2)$ for $k=4$; $(2,3,3,2)$ for $k=5$; $(3,5,6,5,3)$ for $k=6$;
$(3,5,6,6,5,3)$ for $k=7$;\linebreak 
\scalebox{.93}[1]{$(4,7,9,10,9,7,4)$ for $k=8$; 
$(4,7,9,10,10,9,7,4)$ for $k=9$;  and  $(5,9,12,14,15,14,12,9,5)$}\linebreak 
for $k=10$. 

\begin{lemma}
\label{lem:push} 
The polynomial\\
\mbox{}\hfill $p_k(x) = 1 + \lambda^{s_1}x + \lambda^{s_2}x^2 + \cdots +
\lambda^{s_{k-1}}x^{k-1} + x^k$\hfill\mbox{}\\
of degree $k$ is sign-independently real-rooted for any sufficiently large 
value of the positive real parameter $\lambda$.
\end{lemma}

\begin{proof}

This follows from the fact that for large $\lambda$ the polynomial $p_k$ has 
coefficients approaching the polynomial $q_k$ given by:\\
\mbox{}\hfill 
$q_k(x) = (x+\lambda^{-k/2})(x+\lambda^{1-k/2})\cdots(x+\lambda^{k/2})$
\hfill\mbox{}\\
if $k$ is even, and by\\
\mbox{}\hfill $ q_k(x) =
(x+\lambda^{-(k-1)/2})(x+\lambda^{1-(k-1)/2})\cdots(x+\lambda^{(k-1)/2})$
\hfill\mbox{}\\ 
if $k$ is odd.

Indeed, in order to see that $p_k$ is real-rooted for large positive 
$\lambda$, one observes that the
roots of $q_k$ are all real, and since they are given by distinct powers
of $\lambda$, there are $k$ of different magnitude. Hence, under the small
change of real coefficients that is needed to deform $q_k$ to the original
polynomial $p_k$, the roots remain well apart, and hence cannot form any
conjugate pair of complex roots. Now,  one can easily check that for 
sufficiently large $\la$ changing arbitrarily signs of roots of $q_k$ one 
obtains $2^k$ polynomials close to $2^k$ polynomials obtained from $q_k$ by 
arbitrary sign changes of its coefficients. Thus, any change  of signs of 
some of the
coefficients of $p_k$ just corresponds to an appropriate sign change in
some of the roots of $q_k$, and the preceding argument again shows that
the polynomials are still real-rooted.
\end{proof}

\begin{lemma} 
\label{lem:cont} 
The set $\sip_k$ is fibered over $\sip_{k-1}$ with  
contractible $1$-dimensional fibers. 
\end{lemma} 

\begin{proof} Notice that the restriction of $\sip_k$ to the hyperplane 
$a_0=0$ is in obvious $1$-$1$ correspondence with $\sip_{k-1}$ obtained 
by dividing a polynomial $p(x)=a_1x+\cdots+a_kx^k$ from the former set  by the 
variable $x$. To finish the proof we show that for any $p(x)=a_0+a_1x+\cdots
+a_kx^k$ belonging to $\sip_k$ the family of polynomials 
$p_{\tau}=p-a_0\tau,\;\tau\in[0,1]$ belong to $\sip_k$ thus forming 
the required fiber of the projection in question.  Indeed, consider  for any 
real rooted polynomial 
$p(x)=a_0+a_1x+\cdots+a_k x^k$ the family $p_\eps(x)=p(x)+\eps$ where 
$\eps\in \R$. It is obvious that 
$p_\eps(x)$ is real-rooted if and only if $\eps\in [v_{min}, V_{max}]$ where 
$v_{min}$ is the maximal local minimum of $p(x)$ and $V_{max}$ is its minimal 
local maximum. Now take $p\in \sip_k$ and consider its family 
$p_\eps(x)$. Since all the $a_i$ are now nonnegative,  
consider $p_{-}(x)=-a_0+a_1x+\cdots+a_kx^k$ which must also be real-rooted. 
Thus at least for $\eps$ in the interval $[-2a_0,0]$ one has that $p_\eps(x)$ 
is real-rooted. Exactly  the same argument works for all $p_\pm$ obtained 
from $p$ by  arbitrary sign changes of its coefficients proving that the 
family $p-a_0\tau,\; \tau\in [0,1]$ sits inside $\sip_k$. 
\end{proof}

\subsection{Finding Recession Cones} 
Denote by $\Gamma_k$ the connected component of\linebreak $\R^{k+1}\setminus 
\amoeba(\Delta_k)$ corresponding to the monomial (1), and let $C_k$ denote the 
recession cone of $\Gamma_k$. We now prove the following crucial result. 
\begin{lemma} \label{lem:cone}
The cone $C_k$ is given by the inequalities $2x_l\ge x_{l-1}+x_{l+1}$, for 
$l\in\{1,\ldots,k-1\}$.
\end{lemma} 
 
\begin{proof}
Recall that for a polynomial $p({\bf z})$ in $n$ complex variables 
${\bf z}=(z_1,...,z_n)$, one defines its {\em Ronkin function} $N_p(x)$,  
in $n$ real variables $\bar{x}=(x_1,\ldots,x_n)$, by the formula
$$\frac{1}{(2\pi i)^n}\int_{{\Log}^{-1}(\bf x)}\log|p({\bf z})|
\frac{dz_1}{z_1}\wedge \cdots \wedge \frac{dz_n}{z_n},$$
where ${\bf x}=(x_1,\ldots,x_n)$. It is known that the Ronkin function is 
convex, and it is affine on each connected component of the complement of the 
amoeba $\amoeba(p)$. Equivalently, $N_p$ is given by the 
integral $$N_p({\bf x})=\frac{1}{(2\pi)^n}\int_{[0,2\pi]^n}\log|p({\bf z})|
d\theta_1 \cdots d\theta_n,$$
where $${\bf z}=\left(e^{x_1+i\theta_1},\ldots, e^{x_n+i\theta_n}
\right)$$ \cite{PT}. As an example, the Ronkin function of a monomial 
$p({\bf z})=az_1^{l_1}\cdots z_n^{l_n}, \; a\neq 0$ 
is given by $$N_p({\bf x})=\log|a|+l_1x_1+\cdots+l_nx_n.$$
From general results proved in 
\cite{PR} one knows that the Ronkin function of $\Delta_k$ is equal to 
$\log|c_v|+\langle
v,x\rangle$ in the component corresponding to a vertex monomial $c_vx^v.$ 

In particular, in the components of the special vertex monomials (1) and
(2), the Ronkin function coincides with  the affine linear functions\\
\mbox{}\hfill $2x_1+\cdots +2x_{k-1}=2\,(x_1+\cdots+x_{l-1})$\hfill 
\mbox{}\\
and\\
\mbox{}\hfill $x_{l-1}-2x_l+x_{l+1}+2\,(x_1+\cdots+x_{k-1})$\hfill\mbox{}\\ 
respectively. Now, by \cite{PST} one knows that the amoeba of 
$\Delta_k(1,a_1,\ldots,a_{k-1},1)$ does
not have any other unbounded connected components to its complement  
other than those corresponding to the vertices of 
$\newt(\Delta_k(1,a_1,\ldots,a_{k-1},1))$. 
Now let $S_{\Delta_k}$ denote the spine (see \cite{PR} for its 
definition) and let $S_\infty$ denote a sufficiently small neighborhood 
of $S_{\Delta_k}$ about infinity. It then follows that $S_\infty$ is 
exactly a neighborhood about infinity of the corner locus 
of the pieceswise linear convex function (or tropical polynomial)\\
\mbox{}\hfill $\max_v\bigl(\log|c_v|+\langle v,x\rangle\bigr)$,\hfill 
\mbox{}\\
where $v$ ranges over the vertices of the Newton polytope of $\Delta_k$.
The unbounded connected components of the complement of the spine 
$S_{\Delta_k}$ 
are convex polyhedral cones where one of the affine linear functions 
dominates all the others, and the closure of such a cone is the recession 
cone of the unbounded connected component of the complement to 
$\amoeba(\Delta_k)$. 
For the special vertex monomial (1)
we obtain in this way that the recession cone $C_k$ of $\Gamma_k$ is given 
by the inequalities: $$ 2\,(x_1+\cdots+x_{k-1}) \ge  
x_{l-1}-2x_l+x_{l+1}+2\,(x_1+\cdots+x_{k-1}),\, 
l\in\{1,\ldots,k-1\}\,,$$ 
or, equivalently, $2x_l\ge  x_{l-1}+x_{l+1}$, for $l\in\{1,\ldots,k-1\}$.  
\end{proof}

We will later need the following refinement of Lemma \ref{lem:cone}
that characterizes the unique translate $C^s_k$ of $C_k$ {\em supporting}
$\Gamma_k$. 
\begin{lemma}
\label{lem:sup} 
The cone $C^s_k$ defined by the inequalities
$2x_l\geq x_{l-1}+x_{l+1}+\log 4$ for all $l\!\in\!\{1,\ldots,k-1\}$
contains $\Gamma_k$, but $y+C^s_k$ does not contain $\Gamma_k$ for 
any $y$ in the interior of $C_k$. 
\end{lemma} 

\noindent 
{\em Proof.} 
First note that each polynomials $x^m(1+x)^2$, for
$m\!\in\!\{0,\ldots,k-2\}$, lies on a unique facet of the 
cone $C^s_k$, and that this cone has exactly $k-1$ facets. 
So to conclude, we need only show that each such polynomial 
lies on the boundary of $\Gamma_k$. However, the last statement 
was already observed in the introduction, during our discussion 
of perturbing middle coefficients. \qed   

\medskip
\noindent 
{\em Proof of Proposition 1.} From our earlier discussion, we know that 
the set $\sip_k$ (if non-empty) consists of some number of connected 
components of the complement $\R^{k+1}\setminus \Delta^\dagger_k$ where 
$\Delta^\dagger_k$ is the reflected 
discriminant of $\Delta_k$ (see, e.g., Figure 4). Indeed, $\sip_k$ is 
the intersection of the set of all degree $k$ real-rooted polynomials 
having only simple zeros with all similar sets 
obtained by all possible sign changes of the coefficients. By Lemmata 
\ref{lem:push} and \ref{lem:cont}   
the set $\sip_k$ is non-empty and connected, so 
$\sip_k$ coincides with a unique 
connected component of $\R^{k+1}\setminus \Delta^\dagger_k$. 

To conclude, we have to show that the 
image of $\sip_k$ under $\Log|\cdot|$ coincides with the component of the 
complement to $\amoeba(\Delta_k)$ corresponding to the monomial (1). We show 
that the vector $s\in\mathbb{N}^{k-1}$ from Lemma \ref{lem:push} is an interior 
point in the recession cone of the unbounded connected component $\Gamma_k$ 
of the complement of the discriminant amoeba corresponding to the 
finest subdivision of $\{0,\ldots,k\}$. Indeed, this recession cone is
defined by the inequalities $2x_j \ge  x_{j-1} + x_{j+1}$,
$j\!\in\!\{1,\ldots,k-1\}$ with the dehomogenizing convention 
$x_0=x_k=0$, thanks to Lemma \ref{lem:cone}. This
means that the coefficients $\lambda^{s_j}$ of the polynomial $p_k$ from
Lemma \ref{lem:push}, for large enough $\lambda$, represent a point in 
$\Gamma_k$. But the polynomial $p_k$ was seen 
to be sign-independently real-rooted for large $\lambda$, and this concludes
the proof.  \qed 

\section{The Proofs of our Main Results} 

\subsection{Theorem 1} 
The proof of the ``only if'' direction is easy, as outlined in the
introduction: If $T_\ga(\sip)\!\subseteq\!\rr$ then we must 
certainly have $T_\ga(x^m(1+x)^2)\!\in\!\rr$ for all $m$, since 
$x^m(1+x)^2\!\in\!\sip$ for all $m$. Thus, $\ga$ must be log-concave. 

The proof of the ``if'' direction is more intricate but now follows  
easily from our preceding development: By Proposition 1 and 
Lemma \ref{lem:cone}, $\Log|\cdot|$ of the set of log-concave 
$\ga=(\ga_0,\ldots,\ga_k)$ is precisely the recession cone $C_k$ of 
$\Gamma_k$, and\linebreak 
$\Log|\cdot| : \sip_k\longrightarrow \Gamma_k$ is a diffeomorphism. 
So any such $\ga$ satisfies $T_\ga\!\left(\sip_k\right)\!\subseteq\!\sip_k$, 
and we are done. \qed 

\subsection{Corollary 1} 
The first part of the Corollary follows immediately from 
Lemma \ref{lem:sup}. The second part follows easily by 
applying Lemma \ref{lem:cont} inductively. \qed 

\subsection{Theorem 2} 
Our proof here will be completely parallel to that of 
Theorem 1, so let us start with some analogues of 
$\Gamma_k$ and $C_k$: First, let us denote by $\Gamma'_k$ the connected 
component of $\R^{k+1}\setminus
\amoeba(\Delta_k)$ corresponding to the trivial (single-celled) subdivision 
of $\{0,\ldots,k\}$. Also let $C'_k$ denote the recession cone of $\Gamma'_k$. 
\begin{lemma} 
\label{lem:back} 
The cone $C'_k$ is given by the inequalities $kx_j\leq j(x_k-x_0)$, for
$j\in\{1,\ldots,k-1\}$. \qed 
\end{lemma} 

\noindent 
Lemma \ref{lem:back} follows easily from the 
development of \cite{GKZ,PT} just like Lemma 
\ref{lem:cone}, so we proceed to an analogue of Proposition 1:   
\begin{prop} The map 
$\Log|\cdot| : \iip_k \longrightarrow \Gamma'_k$ is a diffeomorphism. \qed  
\end{prop} 

\noindent 
Proposition 2 is proved in exactly the same way as Proposition 1, 
save that one uses a different deformation argument along the way: 
Lemma \ref{lem:push} is replaced by the observation that 
(a) $q_k(x):=1+\lambda^{-1}x+\cdots+\lambda^{-1}x^{k-1}+x^k\!\in\!\iip_k$ 
for all sufficiently large $\lambda$, and (b) the roots of 
$q_k$ approach those of $x^k+1$ as $|\lambda|\rightarrow \infty$. 

\medskip 
We are now ready to prove Theorem 2: \\
{\em Proof of Theorem 2:} 
The ``only if'' direction can be proved as follows: 
For any $j\!\in\!\{1,\ldots,k-1\}$, consider the 
polynomial $p_j(x):=(k-j)-kx^j+jx^k$. 
It is then easily checked that (a) $p_j$ has a unique degenerate 
real root, (b) $p_j$ has exactly $1$ or $2$ real roots according 
as $k$ is even or odd, (c) $p^-_{j,\eps}(x):=(k-j)-k(1-\eps)x^j+jx^k\in
\iip_k$ for all $\eps\!\in\!(0,1]$, and 
(d) $p^+_{j,\eps}(x):=(k-j)-k(1+\eps)x^j+jx^k\not\in\iip_k$ 
for all $\eps\!>\!0$ (To prove (a)--(d) one can simply 
apply Descartes' Rule of Signs and a clever formula for the 
discriminant of a trinomial from \cite[Prop.\ 1.2, pg.\ 217]{GKZ}.)
Thus, should the stated inequalities involving $(\ga_0,\ga_j,\ga_k)$ 
fail to hold, we can easily find an $\eps\!>\!0$ such that 
$T_\ga(p^-_{j,\eps})\!\not\in\!\iip_k$ (with $\ga:=(\ga_0,
\underset{j-1}{\underbrace{1,\ldots,1}},\ga_j,
\underset{k-j-1}{\underbrace{1,\ldots,1}},\ga_k)$) and obtain a \linebreak 

\vspace{-.6cm} 
\noindent  
contradiction. 
 
The proof of the ``if'' direction is more intricate, but follows
easily from our development: By Proposition 2 and Lemma 
\ref{lem:back}, $\Log|\cdot|$ of the set of 
$\ga=(\ga_0,\ldots,\ga_k)$ satisfying the stated inequalities is precisely 
the recession cone $C'_k$ of $\Gamma'_k$,
and\linebreak 
$\Log|\cdot| : \iip_k\longrightarrow \Gamma'_k$ is a diffeomorphism.
So any such $\ga$ satisfies $T_\ga\!\left(\iip_k\right)\!\subseteq\!\iip_k$, 
and we are done. \qed 

\section{Future Directions}

\noindent
{\bf Problem 1.} {\em How does one count connected components of the 
complement to the reflected discriminant of a given discriminant? In 
particular, is it true 
that the number of connected components of the complement to the reflected 
discriminant of univariate polynomials of degree $k$ restricted to $\R_+^k$ 
equals $2^{k}$? } 

\medskip
\noindent
{\bf Problem 2.} {\em Find an elementary proof of Theorem~\ref{th:main} 
avoiding the use of discriminant amoebae.} 

\medskip 
\noindent 
Regarding the last problem, we observe that we first derived our 
characterization of the recession cone relevant to $\sip_k$ via some quick, 
informal calculations using the {\em Horn-Kapranov Uniformization} 
\cite{kapranov,PT}. (The Horn-Kapranov Uniformization is a remarkably useful 
rational parametrization of the $A$-discriminant variety. An intriguing fact 
is that the resulting parametric formula for $H_{\Delta_k}$ has size polynomial 
in $k$, while $\Delta_k$ has a number of monomials (and coefficient 
bit-sizes) exceeding $2^{k-1}$ \cite{bhpr}.) It is likely that the 
Horn-Kapranov Uniformization can yield an alternative proof of 
Theorem \ref{th:main} without Ronkin functions. This could be seen 
as a step toward solving Problem 2. 

\section*{Acknowledgements} 
The second author thanks the Wenner Gren Foundation for the support of
his visit to Stockholm University during the start of this project.
The second author was also partially supported by 
NSF CAREER grant DMS-0349309 and Sandia National Laboratories. We are also 
sincerely grateful to P.\ Br\"anden for finding a mistake in the initial 
version of the paper (and pointing out a number of relevant references), and 
to Jan-Erik Bj\"ork for pointing out the reference \cite{ostrowski}.


\begin{thebibliography}{30}

\bibitem [BHPR10]{bhpr} O.\ Bastani, C.\ Hillar, D.\ Popov, and 
J.\ M.\ Rojas, {\em ``Sums of Squares, Discriminants, and Randomization in 
Sparse Real Root Counting,''} preprint, Texas A\&M University, 2010. 

\bibitem [CC77]{CC2} T.~Craven and G.~Csordas, 
{\em ``Multiplier sequences for fields,''} Illinois J. Math. {\bf 21(4)} 
(1977), pp.\ 801--817.

\bibitem[CC83]{CC1} T.\ Craven and G.\ Csordas, 
{\em ``Location of zeros. I. Real polynomials and entire functions.''} 
Illinois J.\ Math.\ {\bf 27(2)} (1983), pp.\ 244--278. 

\bibitem[CC04]{CC3} T.\ Craven, and G.\ Csordas,  {\em ``Composition theorems, 
multiplier sequences and complex zero decreasing sequences,''} in Value 
distribution theory and related topics, pp.\ 131--166, Adv. Complex Anal. 
Appl., 3, Kluwer Acad. Publ., Boston, MA, 2004.

\bibitem[Cso01]{csordas} G.\ Csordas, {\em ``Complex zero decreasing 
sequences and the Riemann hypothesis II,''} in Analysis and its 
Applications (Begehr, Gilbert, Wong, eds.), pp.\ 121-134, Kluwer Academic 
Publishers, 2001.  

\bibitem[CVV90]{CVV} G.\ Csordas, R.\ Varga, and I.\ Vincze, {\em ``Jensen 
polynomials with applications to the Riemann $\zeta$-function,''} JMAA {\bf 
153(1)} (1990), pp.\ 112--135.

\bibitem[DaPo49]{DaPo} H.\ Davenport, G.\ Pol\'ya, {\em ``On the product of 
two power series,''} Cand. J. Math. {\bf 1}, (1949), pp.\ 1--5.

\bibitem[GKZ94]{GKZ} I.\ Gelfand, M.\ Kapranov, A.\ Zelevinsky, {Discriminants,
Resultants and Multidimensional Determinants}, Reprint of the 1994
edition. Modern Birkh\"auser Classics. Birkh\"auser Boston, Inc., Boston,
MA, 2008. 

\bibitem[Hut23]{Hu}J.\ I.\ Hutchinson, {\em ``On a remarkable class of entire 
functions,''} Trans.\ Amer.\ Math.\  Soc.\ {\bf 25} (1923), pp.\ 325--332.  

\bibitem[Kap91]{kapranov} M.\ Kapranov, 
{\it ``A characterization of A-discriminantal hypersurfaces in terms of the
logarithmic Gauss map,''} Mathematische Annalen, 290, 1991, pp.\ 277--285.

\bibitem[Kar68]{karlin} S.\ Karlin, {\em Total Positivity}, Vol.\ I.\ 
Stanford University Press, Stanford, Calif.\ 1968.  

\bibitem[KoSh06]{KSh} V.~Kostov and B.~Shapiro, {\em ``On the Schur-Szeg\"o 
composition of polynomials,''} C.\ R.\ Math.\ Acad.\ Sci.\ Paris, {\bf 343(2)} 
(2006), pp.\ 81--86. 
 
\bibitem[Lig97]{Li} T.\ Liggett, 
{\em ``Ultra logconcave sequences and negative dependence,''} 
J.\ Combin.\ Theory Ser.\ A {\bf 79(2)} (1997), pp.\ 315--325.

\bibitem[Ost40]{ostrowski} A.\ Ostrowski, {\em ``Recherches sur 
la m\'ethode de Graeffe et les z\'eros 
des polynomes et des s\'eries de Laurent,''} Acta Math.\ 
{\bf 72}, (1940), pp.\ 99--155. 

\bibitem[PR04]{PR} M.~Passare, H.~Rullg\aa{}rd, {\em ``Amoebas, 
Monge-Amp\`ere measures and triangulations of the Newton polytope,''} 
Duke Math.\ J., {\hbox{\bf 121} } (2004), pp.\ 481--507.

\bibitem[PST05]{PST} M.~Passare, T.~Sadykov, A.~Tsikh,  {\em ``Singularities 
of hypergeometric functions in several variables,''} Compos.\ Math.,  
{\bf 141} (2005), pp.\ 787--810.

\bibitem[PT04]{PT} M.~Passare, A.~Tsikh,  {\em ``Algebraic equations and 
hypergeometric series,''} pp.\ 653--672 in ``The legacy of Niels Henrik Abel", 
Springer, Berlin, 2004.

\bibitem[PS14]{PS}  G.~P\'olya and J.\ Schur, {\em ``\"Uber zwei Arten von 
Faktorenfolgen in der Theorie der algebraischen Gleichungen,''} J.\ Reine 
Angew.\ Math.\ {\bf 144} (1914), pp.\ 89--113.

\bibitem[Sur92]{SU} {\em A survey of linear preserver problems,} 
Linear and Multilinear Algebra, {\bf 33 (1992)}, no.\ 1--2, pp.\ 1--129.

\end{thebibliography}
\end{document}